\documentclass[11pt,a4paper]{article}

\usepackage[T1]{fontenc}
\usepackage[utf8]{inputenc}
\usepackage[english]{babel}
\usepackage{lmodern}
\usepackage[hmargin=22.5mm,vmargin=25mm,footskip=12mm]{geometry}
\usepackage{amsmath,amssymb,amsthm}
\usepackage{enumitem}
\usepackage{url}
\usepackage[hidelinks]{hyperref}

\hypersetup{
  pdftitle={On Entries of Spectral Decompositions and Applications},
  pdfauthor={Musung Kang},
  pdfsubject={Spectral graph theory and continuous-time quantum walks},
  pdfkeywords={average mixing matrix, spectral projection, line graph, strongly cospectral vertices, strongly regular graph}
}
\numberwithin{equation}{section}
\newtheorem{theorem}{Theorem}[section]
\newtheorem{lemma}[theorem]{Lemma}
\newtheorem{proposition}[theorem]{Proposition}
\newtheorem{corollary}[theorem]{Corollary}
\theoremstyle{remark}
\newtheorem{remark}[theorem]{Remark}

\newcommand{\tr}{\operatorname{tr}}
\newcommand{\rank}{\operatorname{rank}}
\newcommand{\im}{\operatorname{im}}
\newcommand{\spec}{\operatorname{spec}}
\newcommand{\supp}{\operatorname{supp}}
\newcommand{\Mhat}{\widehat{M}}
\newcommand{\ee}{\mathbf{e}}
\newcommand{\RR}{\mathbb{R}}
\newcommand{\Q}{\mathbb{Q}}
\newcommand{\bx}{\mathbf{x}}

\title{On Entries of Spectral Decompositions and Applications}
\author{Musung Kang\thanks{Seoul National University, Seoul, South Korea
(musung098@snu.ac.kr). Supported by the BK21 program, National Research Foundation of Korea grant A0424-20260100, and by the
Hyunsong Scholarship Foundation.}}
\date{}

\begin{document}
\maketitle

\begin{abstract}
The entries of the spectral idempotents of a graph are its local and crossed local multiplicities. This paper uses these classical quantities in three applications. First, in a graph on $n$ vertices without isolated vertices, every local multiplicity is at most $\displaystyle\frac{n-1}{n}$, and every diagonal entry of the Ces\`aro average mixing matrix is at most $\displaystyle1-\frac{2}{n}+\frac{2}{n^2}$. In both bounds equality holds only for the complete graph. This proves the connected form of the complete graph trace conjecture of Godsil, Guo and Sobchuk together with its equality case, and component and degree refinements follow. The two upper bounds and their attainment by complete graphs also occur in a recent machine checked proof, without a formal characterization of when equality holds. Second, for the line graph of a connected graph $\Gamma$, the average return probability at a bridge with a bipartite component on one side is at least the reciprocal of the rank of the unsigned incidence matrix. Equality holds exactly for $P_2$ and $P_3$. The proof combines local moments with known inverse-incidence formulas. Finally, the paper gives a short proof of Godsil's entry criterion, which says that vertices $u$ and $v$ are strongly cospectral exactly when $\left|(E_\theta)_{uv}\right|=(E_\theta)_{uu}$ for every eigenvalue $\theta$. The criterion recovers the known classification in connected noncomplete strongly regular graphs, where the strongly cospectral pairs are precisely the two-vertex parts of cocktail party graphs.
\end{abstract}



\section{Introduction}\label{sec:intro}
Let $A=\sum_\theta\theta E_\theta$ be the spectral decomposition of the adjacency matrix of a finite simple graph. The entries of $E_\theta$ record walks between individual vertices and determine the Ces\`aro average mixing matrix. Their diagonal entries are the local multiplicities of Fiol and Garriga \cite{FiolGarriga1997}, whose square roots are the angles of Cvetkovi\'c \cite[Chapter~4]{CRS1997}. The off diagonal entries are the crossed local multiplicities of Fiol, Garriga and Yebra \cite{FGY2001}. The spectral moments of these entries differ from the cosine moments in \cite[Theorem~2.1]{KangSRG}, where uniform average mixing on strongly regular graphs under a freely chosen law for the time reduces to affine conditions on three cosine moments of that law. We use this classical framework in three directions.

\begin{enumerate}[label=\textup{(\arabic*)},leftmargin=*]
\item \emph{Graph spectral-entry bounds.}
For a graph on $n$ vertices without isolated vertices, Lemma~\ref{lem:projector} bounds every local weight by $\displaystyle\frac{n-1}{n}$, with equality only in $K_n$. An inequality for squared weights then gives the sharp diagonal and trace bounds for the average mixing matrix (Theorem~\ref{thm:complete}). This proves the connected form of the complete-graph conjecture of Godsil, Guo and Sobchuk \cite[Conjecture~9.1]{GGS2023} together with its equality case. We also derive bounds for all graphs, for individual components, and in terms of local degrees. The degree bound is attained exactly at complete components (Remark~\ref{rem:degree}). The vertex and diagonal bounds, the trace bound for connected graphs and their attainment by $K_n$ also occur in the formal development \cite{PR11860}, without a formal characterization of when equality holds. The relation of Lemma~\ref{lem:projector} to that development and to earlier eigenvector bounds is described after its proof.

\item \emph{Line-graph return bounds.}
For the line graph of a connected graph $\Gamma$, Theorem~\ref{thm:bridge} gives
$\displaystyle\Mhat_{qq}\ge\frac{1}{\rank B}$ at a bridge $q$ with a bipartite component on one side, where $B$ is the unsigned incidence matrix. The inequality follows from a bound by the size of the eigenvalue support and the absence of weight at zero for $B^{\mathsf T}B$. We characterize equality by $\Gamma\in\{P_2,P_3\}$. Equality forces uniform local weights and a simple positive spectrum. Local and global moments then force the bridge to be pendant. Known inverse-incidence formulas reduce the remaining cases to stars and the paw. Proposition~\ref{prop:hypergraph} gives an extension to uniform linear hypergraphs.

\item \emph{A strongly regular graph corollary.}
We give a short proof of the entry criterion stated in Godsil's working notes \cite{GodsilNotes}. It says that
vertices $u,v$ are strongly cospectral precisely when
$\left|(E_\theta)_{uv}\right|=(E_\theta)_{uu}$ for every $\theta$.
Applying it to the entries of the three spectral idempotents recovers the known classification in connected noncomplete strongly regular graphs. There the strongly cospectral pairs are exactly the two-vertex parts of cocktail party graphs (Corollary~\ref{thm:D}). The sources of the criterion and of the classification are identified in Section~\ref{sec:sc}. The contribution there is the direct proof from the entries.
\end{enumerate}

Section~\ref{sec:cesaro} collects the definitions and notation. Sections~\ref{sec:complete}, \ref{sec:bridge}, and~\ref{sec:sc} develop the three applications. The upper-bound argument uses the neighbor structure of the graph, the bridge argument also uses global traces and inverse Gram diagonals, and the strongly regular case compares diagonal and off diagonal entries. The elementary support-size lower bound is recorded in Theorem~\ref{thm:diag}. The inverse-incidence formulas are recalled with their sources in Lemma~\ref{lem:inverse}. Section~\ref{sec:sc} also gives the rational form of the entry criterion (Corollary~\ref{cor:rational}).

\section{Definitions and notation}\label{sec:cesaro}
All graphs are finite, simple, undirected, and unweighted. For a graph $G$ with vertex set $V$, write $A=A(G)$ for its adjacency matrix, $u\sim v$ for adjacency, $N(u)$ for the set of neighbors of $u$, $d(u)=|N(u)|$ for the degree of $u$, and $\Delta(G)$ for the maximum degree. A vertex is isolated when its degree is zero. Write $V(G)$ and $E(G)$ for the vertex and edge sets of $G$, and $G\setminus v$ and $G-e$ for the graphs obtained by deleting a vertex $v$ and an edge $e$. Write $K_n$, $P_n$ and $C_n$ for the complete graph, the path and the cycle on $n$ vertices, $K_{1,m}$ for the star with $m$ edges, $G_1\cup G_2$ for a disjoint union, and $nK_1$ for the edgeless graph on $n$ vertices. Write $\ee_v$ for the standard basis vector at $v$, $\mathbf 1$ for the all-ones vector, and $J_n$, or $J$, for the all-ones matrix of order $n$. A partition of $V$ is equitable when, for any two cells $X$ and $Y$, every vertex of $X$ has the same number of neighbors in $Y$.

For a graph $\Gamma$ with $m$ edges, the unsigned vertex-edge incidence matrix $B$ has, in the column of an edge, a one at each endpoint and zeros elsewhere. The line graph $L(\Gamma)$ has the edges of $\Gamma$ as vertices, two of them adjacent when they share an endpoint. With the diagonal degree matrix $D(\Gamma)$,
\[
B^{\mathsf T}B=2I_m+A(L(\Gamma)),\qquad BB^{\mathsf T}=D(\Gamma)+A(\Gamma),
\]
by \cite[Lemmas~8.2.2 and~8.2.3]{GR2001}. The matrix $BB^{\mathsf T}$ is the signless Laplacian of $\Gamma$, and $D(\Gamma)-A(\Gamma)$ is its Laplacian. A finite $k$-uniform linear hypergraph $\mathcal H$, with $k\ge2$, has hyperedges of $k$ vertices, and two distinct hyperedges meet in at most one vertex. Its line graph $L(\mathcal H)$ has the hyperedges as vertices, two of them adjacent when they intersect. Its vertex-hyperedge incidence matrix $B_{\mathcal H}$ satisfies $B_{\mathcal H}^{\mathsf T}B_{\mathcal H}=kI+A(L(\mathcal H))$. A graph is strongly regular with parameters $(n,k,\lambda,\mu)$ when it is $k$ regular on $n$ vertices, adjacent vertices have $\lambda$ common neighbors, and distinct nonadjacent vertices have $\mu$ common neighbors. It is primitive when both it and its complement are connected. A connected graph is distance regular when, for vertices $u$ and $v$ at distance $i$, the numbers of neighbors of $v$ at distances $i-1$, $i$ and $i+1$ from $u$ depend only on $i$. The cocktail party graph $K_{2,\ldots,2}$ is the complete multipartite graph whose parts have two vertices each. A signed graph carries a sign $\pm1$ on each edge, and its signed adjacency matrix has these signs in place of the ones of $A$. Switching at a vertex set $X$ negates the signs of the edges between $X$ and $V\setminus X$, and $-K_n$ is $K_n$ with all signs negative.

The spectral decomposition of $A$ over its distinct eigenvalues is
\[
A=\sum_{\theta\in\spec(A)}\theta E_\theta,\qquad
E_\theta E_\eta=\delta_{\theta\eta}E_\theta,\qquad
\sum_\theta E_\theta=I,
\]
where the $E_\theta$ are real orthogonal projections. The multiplicity of $\theta$ is $m_\theta=\tr E_\theta$, the eigenvalue support of a vector $z$ is $\supp z=\{\theta\colon E_\theta z\ne0\}$, and $\varphi(G,t)=\det(tI-A(G))$ is the characteristic polynomial, with derivative $\varphi'(G,t)$ in $t$. The same notation is used for a Hermitian matrix $H$, called a Hamiltonian, in place of $A$, with complex projections. The identities \eqref{eq:weights} and \eqref{eq:moments} below then hold with $H^j$ in place of $A^j$ and $\left|(E_\theta)_{uv}\right|^2$ in place of $(E_\theta)_{uv}^{\,2}$. The continuous-time quantum walk of $H$, its mixing matrix, and its Ces\`aro average are
\[
U_H(t)=\exp(itH),\qquad M_H(t)=U_H(t)\circ\overline{U_H(t)},\qquad
\Mhat_H=\lim_{T\to\infty}\frac1T\int_0^T M_H(t)\,dt,
\]
where $\circ$ denotes entrywise multiplication. The spectral formula $\Mhat_H=\sum_\theta E_\theta\circ\overline{E_\theta}$, standard for graphs \cite{Godsil2013}, holds because the time average of $\exp(it(\theta-\eta))$ vanishes for $\theta\ne\eta$. Replacing $H$ by $H+cI$, where $c\in\mathbb R$, multiplies $U_H(t)$ by the scalar $\exp(itc)$, so it changes neither the projections nor $M_H(t)$ nor $\Mhat_H$. For a graph we write $\Mhat_A(G)=\Mhat_{A(G)}$, or $\Mhat$ when the graph is clear, and $\Mhat_{vv}$ is the average return probability at $v$.

For vertices $u,v$ and every nonnegative integer $j$, the local weights $p_\theta(v)$ and the entries of the projections satisfy
\begin{equation}\label{eq:weights}
p_\theta(v)=(E_\theta)_{vv},\qquad
p_\theta(v)\ge0,\qquad
\sum_\theta p_\theta(v)=1,
\end{equation}
\begin{equation}\label{eq:moments}
(A^j)_{uv}=\sum_\theta\theta^j(E_\theta)_{uv},
\qquad
\widehat M_{uv}=\sum_\theta(E_\theta)_{uv}^{\,2}.
\end{equation}
The diagonal entries $p_\theta(v)$ are the local multiplicities and the off diagonal entries $(E_\theta)_{uv}$ are the crossed local multiplicities \cite[Section~1.2]{DalfoFiol2020}.

Vertices $u$ and $v$ are cospectral when $(E_\theta)_{uu}=(E_\theta)_{vv}$ for every $\theta$, equivalently $\varphi(G\setminus u,t)=\varphi(G\setminus v,t)$ \cite[Theorem~3.1]{GS}, and they are strongly cospectral when
\begin{equation}\label{eq:sc}
E_\theta\ee_u=\pm E_\theta\ee_v\qquad\text{for every }\theta,
\end{equation}
where the sign may depend on $\theta$ \cite{GS}.

For $n\ge2$, the spectral projections of $K_n$ are $\displaystyle\frac1nJ_n$ and $\displaystyle I_n-\frac1nJ_n$, so
\begin{align}\label{eq:complete}
\Mhat_A(K_n)=\left(1-\frac2n\right)I_n+\frac{2}{n^2}J_n,\qquad
\tr\Mhat_A(K_n)=n-2+\frac2n.
\end{align}

\section{Graph spectral-entry bounds}\label{sec:complete}

Godsil, Guo and Sobchuk \cite[Conjecture~9.1]{GGS2023} propose that $K_n$ maximizes the trace of the adjacency average mixing matrix. We use the connected formulation, which is also stated in \cite[Section~8]{MTXZ2026}. Among connected graphs of order $n$, the maximum is attained by $K_n$. A restriction is necessary, since $\Mhat_A(nK_1)=I_n$ has trace $n$. The regular-graph statement \cite[Corollary~6.3]{GGS2023} likewise fails for $nK_1$ when $n>1$. For a connected $k$-regular graph, however, $\displaystyle E_k=\frac1nJ$ gives directly
\[
\Mhat_A(G)_{vv}
=\frac1{n^2}+\sum_{\theta\ne k}p_\theta(v)^2
\le\frac1{n^2}+\left(1-\frac1n\right)^2,
\]
and summing over $v$ gives \cite[Corollary~6.3]{GGS2023} for connected regular graphs. The proof in \cite{GGS2023} passes through \cite[Lemmas~6.1 and~6.2]{GGS2023}, which fail for general equitable partitions. For $K_4$ with the equitable partition into two pairs, the eigenspace of $K_4$ for the eigenvalue $-1$ has dimension three and lies in no eigenspace of $2K_2$, contrary to \cite[Lemma~6.1]{GGS2023}. For the same partition \cite[Lemma~6.2]{GGS2023} gives the bound $2$, whereas $\displaystyle\tr\Mhat_A(K_4)=\frac52$ by \eqref{eq:complete}. The Laplacian analogue is \cite[Corollary~5.2]{GGS2023}. This section proves the adjacency bound without regularity and treats isolated vertices separately.

The first step controls every local spectral weight without any
regularity hypothesis.

\begin{lemma}[A local spectral-weight bound]\label{lem:projector}
Let $G$ be a graph on $n\ge2$ vertices without isolated vertices. For every eigenvalue
$\theta$ of $A(G)$ and every vertex $v$,
\[
(E_\theta)_{vv}\le\frac{n-1}{n}.
\]
If equality holds for some $\theta$ and $v$, then $G\cong K_n$.
\end{lemma}

\begin{proof}
Fix $\theta$ and $v$, and write $E=E_\theta$ and $p=E_{vv}$.
For any vertex $u$, set
\[
y=(A-\theta I)\ee_u.
\]
Since $Ey=0$, the Cauchy--Schwarz inequality gives
\begin{align}\label{eq:projector-cs}
|y_v|^2
=\left|\langle (I-E)\ee_v,y\rangle\right|^2
\le (1-p)\|y\|^2.
\end{align}
Here $\|(I-E)\ee_v\|^2=1-p$. The absence of loops and the unit edge weights give
\begin{align*}
\|y\|^2=d(u)+\theta^2.
\end{align*}

Suppose first that $|\theta|\ge1$. Taking $u=v$ in
\eqref{eq:projector-cs}, we obtain
\[
1-p\ge\frac{\theta^2}{d(v)+\theta^2}
\ge\frac{\theta^2}{n-1+\theta^2}
\ge\frac1n.
\]
If $|\theta|\le1$, choose a neighbor $u$ of $v$, which exists because
$v$ is not isolated. Now $y_v=1$, so
\[
1-p\ge\frac1{d(u)+\theta^2}\ge\frac1n.
\]
This proves the bound.

Assume $\displaystyle p=\frac{n-1}{n}$. The first displayed chain is strict at its last
step when $|\theta|>1$, and the second chain is strict when
$|\theta|<1$. Therefore $\theta^2=1$. Using $u=v$ in
\eqref{eq:projector-cs} yields
\[
\frac1n=1-p\ge\frac1{d(v)+1}\ge\frac1n,
\]
hence $d(v)=n-1$. For every neighbor $u$ of $v$, the same argument
with $y_v=1$ yields $d(u)=n-1$. Since $v$ is adjacent to all other
vertices, every vertex has degree $n-1$, and $G\cong K_n$.
\end{proof}

For a unit eigenvector $x$ with eigenvalue $\theta$, Cioab\u{a} and Gregory \cite[Theorem~3.2]{CioabaGregory2007} proved
$\displaystyle x_v^2\le\frac{d(v)}{d(v)+\theta^2}$ for the principal eigenvector, and Van Mieghem \cite[equation~(18)]{VanMieghem2016} proved it for every eigenvalue. Taking
$\displaystyle x=\frac{E_\theta\ee_v}{\|E_\theta\ee_v\|}$ when $E_\theta\ee_v\ne0$ gives
\[
(E_\theta)_{vv}\le\frac{d(v)}{d(v)+\theta^2}.
\]
This right side is at most $\displaystyle\frac{n-1}{n}$ precisely when $\displaystyle\theta^2\ge\frac{d(v)}{n-1}$, in particular when $|\theta|\ge1$. When $\displaystyle\theta^2<\frac{d(v)}{n-1}$ it exceeds $\displaystyle\frac{n-1}{n}$, and at $\theta=0$ it equals one. The neighbor argument above covers this remaining range. For the Laplacian, Mohan, Tamon, Xu and Zhan \cite[Lemma~6.1]{MTXZ2026} bound the diagonal entries of partial sums of spectral idempotents by a related Cauchy--Schwarz argument over the neighbors of a vertex.

The Cauchy--Schwarz inequalities at the vertex and at a neighbor also occur in \cite{PR11860}, with the degrees bounded by $n-1$. That development eliminates $\theta$ to prove the uniform vertex bound, derives the diagonal bound by the squared-weight inequality of Lemma~\ref{lem:cap}, and proves the connected trace bound and attainment by $K_n$. Its local bounds apply at vertices with a neighbor, without a connectivity assumption. Its notes observe the extension to graphs without isolated vertices and give an informal sketch of why equality in the trace bound forces $K_n$. Neither \cite{PR11860} nor \cite{CioabaGregory2007,VanMieghem2016} states the equality case of Lemma~\ref{lem:projector}, where equality at one vertex and one eigenvalue already forces $G\cong K_n$.

\begin{lemma}[Squared weights under a cap]\label{lem:cap}
Let $p_1,\ldots,p_s$ be nonnegative numbers with sum one, and suppose
$p_i\le\alpha$ for every $i$, where $\displaystyle\frac{1}{2}\le\alpha<1$. Then
\[
\sum_{i=1}^s p_i^2\le\alpha^2+(1-\alpha)^2.
\]
If $\displaystyle \alpha>\frac{1}{2}$, equality holds precisely when the positive entries are
$\alpha$ and $1-\alpha$. If $\displaystyle \alpha=\frac{1}{2}$, equality holds precisely
when there are two positive entries, both equal to $\displaystyle\frac{1}{2}$.
\end{lemma}

\begin{proof}
Let $s_0=\max_i p_i$. If $\displaystyle s_0\le\frac{1}{2}$, then
$\displaystyle\sum_i p_i^2\le s_0\sum_i p_i=s_0\le\frac{1}{2}\le\alpha^2+(1-\alpha)^2$,
because $x^2+(1-x)^2$ has minimum value $\displaystyle\frac{1}{2}$ at $\displaystyle x=\frac{1}{2}$.
If $\displaystyle s_0\ge\frac{1}{2}$, separating an entry equal to $s_0$ gives
\[
\sum_i p_i^2\le s_0^2+(1-s_0)^2
\le\alpha^2+(1-\alpha)^2,
\]
because $x^2+(1-x)^2$ is increasing on $\displaystyle \left[\frac{1}{2},1\right]$.
For $\displaystyle \alpha>\frac{1}{2}$, the first case cannot be tight, because
$\displaystyle\alpha^2+(1-\alpha)^2>\frac{1}{2}$. Equality then requires $s_0=\alpha$ and exactly one
positive entry among the remaining entries. For $\displaystyle \alpha=\frac{1}{2}$,
equality in $\displaystyle\sum_i p_i^2\le\frac{1}{2}\sum_i p_i$ requires each positive
entry to equal $\displaystyle\frac{1}{2}$.
\end{proof}

\begin{theorem}[Complete-graph diagonal and trace bounds]\label{thm:complete}
Let $G$ be a graph on $n\ge2$ vertices without isolated vertices. For every vertex $v$,
\begin{align}\label{eq:diagonal-upper}
\Mhat_A(G)_{vv}\le 1-\frac2n+\frac{2}{n^2}.
\end{align}
Equality holds at some vertex if and only if $G\cong K_n$.
In particular,
\begin{align}\label{eq:trace-upper}
\tr\Mhat_A(G)\le n-2+\frac2n
=\tr\Mhat_A(K_n),
\end{align}
with equality if and only if $G\cong K_n$. Thus the connected form of
\cite[Conjecture~9.1]{GGS2023} holds for every $n\ge2$, and it holds for $n=1$ because
$K_1$ is the only graph of order one.
\end{theorem}

\begin{proof}
Apply Lemma~\ref{lem:cap} to the weights in \eqref{eq:weights}, whose squares sum to $\Mhat_A(G)_{vv}$ by \eqref{eq:moments}, using Lemma~\ref{lem:projector} and the cap $\displaystyle\alpha=\frac{n-1}{n}$. This gives
\[
\Mhat_A(G)_{vv}
\le\left(\frac{n-1}{n}\right)^2+\left(\frac1n\right)^2,
\]
which is \eqref{eq:diagonal-upper}.

For $n\ge3$, equality implies that some local spectral weight is
$\displaystyle\frac{n-1}{n}$. Lemma~\ref{lem:projector} then gives $G\cong K_n$.
For $n=2$, the only graph without isolated vertices is $K_2$. Conversely,
\eqref{eq:complete} verifies equality at every vertex of $K_n$.
Summing \eqref{eq:diagonal-upper} over the $n$ vertices proves
\eqref{eq:trace-upper}. If equality holds in the trace bound,
every diagonal bound is an equality, so $G\cong K_n$.
\end{proof}

The spectral projections of a disjoint union are block diagonal, so its average mixing matrix is the direct sum of those of its components.

\begin{corollary}[All graphs of a given order]\label{cor:allgraphs}
Let $G$ be a graph on $n\ge2$ vertices. Then $\tr\Mhat_A(G)\le n$,
with equality if and only if $G$ has no edges. If $G$ has at least one
edge, then $\tr\Mhat_A(G)\le n-1$, with equality if and only if
$G\cong K_2\cup(n-2)K_1$.
\end{corollary}

\begin{proof}
By \eqref{eq:weights} and \eqref{eq:moments}, $\Mhat_A(G)_{vv}=\sum_\theta p_\theta(v)^2\le\sum_\theta p_\theta(v)=1$,
with equality if and only if one weight equals one, that is, if and
only if $\ee_v$ is an eigenvector of $A$. Since $(A\ee_v)_v=0$, this
happens exactly when $A\ee_v=0$, that is, when $v$ is isolated. Summing
over $v$ proves the first statement. For the second, let $G'$ be the
graph obtained from $G$ by deleting its $n-n'$ isolated vertices, so
that $G'$ has $n'\ge2$ vertices and no isolated vertices. The isolated
vertices contribute $n-n'$ to the trace, and
Theorem~\ref{thm:complete} applied to $G'$ gives
\[
\tr\Mhat_A(G)\le n-n'+n'-2+\frac{2}{n'}=n-2+\frac{2}{n'}\le n-1,
\]
with equality throughout if and only if $n'=2$ and $G'\cong K_2$.
\end{proof}

\begin{corollary}[Component form]\label{cor:components}
Let $G$ be a graph whose components have orders $n_1,\ldots,n_c$,
isolated vertices included as components of order one. Then
\[
\tr\Mhat_A(G)\le\sum_{i=1}^{c}\left(n_i-2+\frac{2}{n_i}\right),
\]
with equality if and only if every component of $G$ is complete.
\end{corollary}

\begin{proof}
Apply Theorem~\ref{thm:complete} to each component of order at least two. A component $K_1$ contributes $1$, which is the value of $\displaystyle n_i-2+\frac{2}{n_i}$ at $n_i=1$. Equality holds exactly when every component is complete.
\end{proof}

This is the adjacency analogue of \cite[Corollary~5.2]{GGS2023}.

\begin{remark}[A degree dependent refinement]\label{rem:degree}
The proof of Lemma~\ref{lem:projector} only uses $n-1$ as an upper
bound for $d(v)$ and for the degree of one neighbor of $v$. Let $v$ be nonisolated and put
\[
D_v=\max\left\{d(v),\ \min_{u\sim v}d(u)\right\}\le\Delta(G).
\]
The two displayed chains
give $\displaystyle 1-p\ge\frac{1}{d(v)+1}$ when $|\theta|\ge1$ and
$\displaystyle 1-p\ge\frac{1}{d(u)+1}$ for every neighbor $u$ when
$|\theta|\le1$. For $|\theta|\ge1$ the first of these also follows from the
inequality of Van Mieghem discussed after Lemma~\ref{lem:projector}. Hence
\[
(E_\theta)_{vv}\le1-\frac{1}{D_v+1}
\qquad\text{and}\qquad
\Mhat_A(G)_{vv}\le1-\frac{2}{D_v+1}+\frac{2}{(D_v+1)^2},
\]
where the second bound follows from Lemma~\ref{lem:cap} with
$\displaystyle\alpha=\frac{D_v}{D_v+1}\ge\frac{1}{2}$. This is never weaker than
\eqref{eq:diagonal-upper}, because $D_v\le n-1$ and $\displaystyle\frac{x^2+1}{(x+1)^2}$
increases for $x\ge1$. It is much stronger when $D_v$ is small compared with $n$,
for instance for graphs of bounded maximum degree. Since the entries of the
idempotents at $v$ only depend on the component of $v$, both bounds
and \eqref{eq:diagonal-upper} also hold with that component in place
of $G$ whenever $v$ is not isolated.

Equality in the first bound for some $\theta$, or in the second bound,
holds precisely when the component of $v$ is the complete graph
$K_{D_v+1}$. By the equality case of
Lemma~\ref{lem:cap}, equality in the second bound forces
$\displaystyle(E_\theta)_{vv}=\frac{D_v}{D_v+1}$ for some $\theta$, so it
suffices to treat the first bound. Suppose that $p=(E_\theta)_{vv}$
equals $\displaystyle\frac{D_v}{D_v+1}$. If $|\theta|>1$, the chain with
$u=v$ gives $\displaystyle 1-p\ge\frac{\theta^2}{d(v)+\theta^2}>\frac{1}{d(v)+1}$.
If $|\theta|<1$, the chain with a neighbor $u$ of minimum degree gives
$\displaystyle 1-p\ge\frac{1}{d(u)+\theta^2}>\frac{1}{d(u)+1}$. Both
contradict $\displaystyle 1-p=\frac{1}{D_v+1}$, so $\theta^2=1$. The chain
with $u=v$ then forces $d(v)=D_v$ and equality in \eqref{eq:projector-cs}.
Hence $(I-E_\theta)\ee_v=c\,y$ for $y=(A-\theta I)\ee_v$ and a scalar $c\ne0$,
and applying $A-\theta I$ gives $y=c\,(A-\theta I)y$. A vertex $w$ at
distance two from $v$ would have $y_w=0$ but
$\left((A-\theta I)y\right)_w=|N(w)\cap N(v)|>0$. So the component of
$v$ consists of $v$ and its $D_v$ neighbors, and Lemma~\ref{lem:projector}
applied to this component shows that it is $K_{D_v+1}$. Conversely, at a
vertex of $K_{D_v+1}$ the eigenvalue $-1$ has weight
$\displaystyle\frac{D_v}{D_v+1}$ and the eigenvalue $D_v$ has weight
$\displaystyle\frac{1}{D_v+1}$, so both bounds are attained.
\end{remark}

\begin{remark}[Other Hamiltonians]\label{rem:hamiltonian}
The proof of Lemma~\ref{lem:projector} uses three properties of $A$.
The diagonal of $A$ vanishes, so that
$\|(A-\theta I)\ee_u\|^2=\|A\ee_u\|^2+\theta^2$. The squared row norms
$\|A\ee_u\|^2=d(u)$ are at most $n-1$. Finally $|A_{vu}|=1$ for a
neighbor $u$ of $v$. Let $H$ be a Hermitian matrix with zero diagonal,
let $v$ be an index for which some $u\ne v$ has $H_{vu}\ne0$, and put
$w=|H_{vu}|$ and $R=\max\left\{\|H\ee_v\|^2,\|H\ee_u\|^2\right\}$. The same
two chains give
\[
(E_\theta)_{vv}\le\frac{R}{R+w^2}
\]
for every eigenvalue $\theta$ of $H$. Indeed, if $\theta^2\ge w^2$ the
chain with $u=v$ gives
$\displaystyle 1-(E_\theta)_{vv}\ge\frac{\theta^2}{\|H\ee_v\|^2+\theta^2}\ge\frac{w^2}{R+w^2}$,
and if $\theta^2\le w^2$ the chain with $u$ gives
$\displaystyle 1-(E_\theta)_{vv}\ge\frac{w^2}{\|H\ee_u\|^2+\theta^2}\ge\frac{w^2}{R+w^2}$.
Since $R\ge w^2$, Lemma~\ref{lem:cap} applies with
$\displaystyle\alpha=\frac{R}{R+w^2}$ and bounds the diagonal entry of the
average mixing matrix of $H$ at $v$ by $\alpha^2+(1-\alpha)^2$. The
matrices $H$ and $H+cI$ have the same idempotents and the same average
mixing matrix, so a constant diagonal may replace the zero diagonal.
For a signed graph, or for a graph with equal loop weights at all
vertices, this recovers $\displaystyle\frac{n-1}{n}$ and
$\displaystyle1-\frac{2}{n}+\frac{2}{n^2}$, and the Laplacian and the signless
Laplacian of a regular graph are covered. For a signed graph the equality
statement of Lemma~\ref{lem:projector} carries over, but its converse does not.
Equality still forces $\theta^2=1$ and the underlying graph to be $K_n$. The
equality cases of the Cauchy--Schwarz steps make every column of $H-\theta I$
parallel to $(I-E_\theta)\ee_v$, so $H=\theta(I-\bx\bx^{\mathsf T})$ with
$\bx\in\{1,-1\}^n$, and the signed graph is switching equivalent to $K_n$ or
to $-K_n$. Both attain equality. A signed complete graph need not attain it.
For $K_5$ with the five edges of a Hamilton cycle negative, the spectrum is
$0$ and $\pm\sqrt5$ with multiplicity two. Its signed adjacency matrix $H$ has $H_{vv}=0$ and $(H^2)_{vv}=4$ at every vertex $v$, so the moments \eqref{eq:moments} with $j=0,1,2$ give the local weights $\displaystyle\frac{1}{5}$ at $0$ and $\displaystyle\frac{2}{5}$ at $\sqrt5$ and at $-\sqrt5$. Hence every diagonal entry of an
idempotent is at most $\displaystyle\frac{2}{5}$.
\end{remark}

The weights in \eqref{eq:weights} also give a lower bound for the diagonal entries of the average mixing matrix. Unlike Theorem~\ref{thm:complete}, it needs no hypothesis on isolated vertices, and its Cauchy--Schwarz step reappears in Section~\ref{sec:bridge}.

\begin{theorem}[A diagonal lower bound]\label{thm:diag}
Let $G$ be a graph on $n\ge2$ vertices, connected or not, let $v$ be a vertex, and write $\Mhat=\Mhat_A(G)$. Then
\[
\Mhat_{vv}\ge\frac{1}{|\supp\ee_v|}\ge\frac{1}{n}.
\]
Moreover $\displaystyle\Mhat_{vv}=\frac{1}{n}$ holds if and only if $A$ has $n$ distinct eigenvalues and $\displaystyle(E_\theta)_{vv}=\frac{1}{n}$ for every $\theta$, equivalently $\varphi(G,t)$ has $n$ distinct roots and
$n\,\varphi(G\setminus v,t)=\varphi'(G,t)$.
\end{theorem}

\begin{proof}
Since $p_\theta(v)=\|E_\theta\ee_v\|^2$, the weight $p_\theta(v)$ is positive exactly when $\theta\in\supp\ee_v$. By the Cauchy--Schwarz inequality on these $|\supp\ee_v|$ positive weights,
\[
1=\left(\sum_{\theta\in\supp\ee_v}p_\theta(v)\right)^2\le |\supp\ee_v|\sum_{\theta\in\supp\ee_v}p_\theta(v)^2=|\supp\ee_v|\,\Mhat_{vv}.
\]
The number $|\supp\ee_v|$ is at most the number of distinct eigenvalues, hence at most $n$.
Equality $\displaystyle\Mhat_{vv}=\frac1n$ forces $|\supp\ee_v|=n$, so all $n$ eigenvalues are distinct and have positive weight at $v$, and equality in the Cauchy--Schwarz inequality forces all weights to be equal to $\displaystyle\frac1n$.
Conversely, if $A$ has $n$ distinct eigenvalues and every weight equals $\displaystyle\frac1n$, then $\displaystyle\Mhat_{vv}=n\left(\frac1n\right)^2=\frac1n$ by \eqref{eq:moments}.
For the polynomial form, Cramer's rule and the spectral decomposition give
\[
\frac{\varphi(G\setminus v,t)}{\varphi(G,t)}=\left((tI-A)^{-1}\right)_{vv}=\sum_\theta\frac{p_\theta(v)}{t-\theta},
\qquad
\frac{\varphi'(G,t)}{\varphi(G,t)}=\sum_\theta\frac{m_\theta}{t-\theta}.
\]
By the uniqueness of partial fraction expansions, $n\,\varphi(G\setminus v,t)=\varphi'(G,t)$ holds exactly when $n\,p_\theta(v)=m_\theta$ for every $\theta$. When $\varphi(G,t)$ has $n$ distinct roots, every $m_\theta$ equals $1$ and this says that every weight equals $\displaystyle\frac1n$. The simple spectrum cannot be dropped. Every vertex $v$ of $K_3$ has $3\,\varphi(K_3\setminus v,t)=3t^2-3=\varphi'(K_3,t)$, while $\displaystyle\Mhat_A(K_3)_{vv}=\frac59$.
\end{proof}

Summing the final bound $\displaystyle\Mhat_{vv}\ge\frac1n$ in Theorem~\ref{thm:diag} gives
$\tr\Mhat_A(G)\ge1$. Coutinho and Godsil \cite[Lemma~15.3.2]{CGbook} prove this trace bound and show that equality forces $n\le2$. Hence for $n\ge3$ the value $\displaystyle\Mhat_{vv}=\frac1n$ cannot occur at every vertex. The question of attainment at a single vertex is discussed in Remark~\ref{rem:hyperequality}.

\section{Line-graph return bounds and bridge equality}\label{sec:bridge}

Let $\Gamma$ be a connected graph on $n$ vertices with $m$ edges and unsigned
incidence matrix $B\in\{0,1\}^{n\times m}$. Edges are denoted by letters such as $q$ and $f$, and
$\ee_q\in\RR^m$ denotes the coordinate vector of the edge $q$. By Section~\ref{sec:cesaro},
\begin{align*}
C=B^{\mathsf T}B=2I_m+A(L(\Gamma)),\qquad
Q=BB^{\mathsf T}=D+A(\Gamma),
\end{align*}
where $D=D(\Gamma)$ and $Q$ is the signless Laplacian.
Since the two Hamiltonians $C$ and $A(L(\Gamma))$ differ by $2I_m$,
we write
\[
\Mhat=\Mhat_A(L(\Gamma))=\Mhat_C.
\]
Thus $\Mhat_{qq}$ is an average return probability at an edge of
$\Gamma$. Average mixing matrices of trees, where every edge is a
bridge, were studied by Godsil, Guo and Sinkovic \cite{GGSink2018}.

\begin{theorem}[Bridge bound on line graphs]
\label{thm:bridge}
Let $\Gamma$ be a connected graph on $n\ge2$ vertices, with unsigned
incidence matrix $B$. Suppose $q$ is a bridge and at least one
component of $\Gamma-q$ is bipartite. Then
\begin{align}\label{eq:bridge-bound}
\Mhat_{qq}\ge\frac1{\rank B}
=
\begin{cases}
\dfrac1{n-1},&\Gamma\text{ is bipartite},\\[4pt]
\dfrac1n,&\Gamma\text{ is nonbipartite}.
\end{cases}
\end{align}
Moreover,
\begin{align*}
\Mhat_{qq}=\frac1{\rank B}
\quad\Longleftrightarrow\quad
\Gamma\in\{P_2,P_3\}.
\end{align*}
Consequently the inequality is strict for every other graph satisfying
the hypotheses.
\end{theorem}

Lemma~\ref{lem:bridge-support} proves the inequality. The proof of the equality classification follows Lemma~\ref{lem:inverse}.

\begin{lemma}[Rank of the unsigned incidence matrix]\label{lem:rank}
For a connected graph $\Gamma$ on $n$ vertices,
\[
\rank B=
\begin{cases}
n-1,&\Gamma\text{ is bipartite},\\
n,&\Gamma\text{ is nonbipartite}.
\end{cases}
\]
\end{lemma}

\begin{proof}
This is a classical result, see \cite[Theorem~8.2.1]{GR2001}. A vector $z$ lies in
$\ker B^{\mathsf T}$ exactly when $z_u+z_v=0$ for every edge $uv$. In a
connected bipartite graph the solutions are the scalar multiples of
the vector taking the values $1$ and $-1$ on the two parts. If the
graph has an odd cycle, alternation around that cycle forces $z$ to
vanish on the cycle, and connectedness then forces $z=0$. This
description of $\ker B^{\mathsf T}$ is used below.
\end{proof}

\begin{lemma}[Bridge support and the rank bound]\label{lem:bridge-support}
Under the hypotheses of Theorem~\ref{thm:bridge}, put $r=\rank B$.
Then $\ee_q\in\im B^{\mathsf T}$ and $\displaystyle\Mhat_{qq}\ge\frac{1}{r}$.
If equality holds, then $m=r$, all eigenvalues of $C$ are positive
and simple, and their local weights at $q$ are all $\displaystyle\frac{1}{m}$.
In particular, $\Gamma$ is either a tree or a unicyclic graph whose
unique cycle has odd length.
\end{lemma}

\begin{proof}
Let $S$ be the vertex set of a bipartite component of $\Gamma-q$.
Give its two parts values $1$ and $-1$, and give all vertices outside
$S$ value zero. The resulting vector $z$ satisfies
$B^{\mathsf T}z=\pm\ee_q$, because $q$ is the only edge leaving
$S$. Thus $\ee_q\in\im B^{\mathsf T}=(\ker B)^\perp$.

Let $F_\lambda$ be the spectral projection of $C$ at $\lambda$, and
write $p_\lambda=(F_\lambda)_{qq}$, with $F_0=0$ if $0$ is not an
eigenvalue. Since
$\ker C=\ker B$, we have $p_0=0$. If $s$ positive eigenvalues have
$p_\lambda>0$, then $s\le r$ and
\begin{align}\label{eq:rank-cs}
\Mhat_{qq}=\sum_{\lambda>0}p_\lambda^2
\ge\frac{\left(\sum_{\lambda>0}p_\lambda\right)^2}{s}
=\frac1s\ge\frac1r.
\end{align}
For $m\ge2$, the first inequality is the support-size bound of Theorem~\ref{thm:diag}, since $C$ and $A(L(\Gamma))$ have the same projections. The incidence condition removes the weight at $\lambda=0$ for $C$, equivalently at eigenvalue $-2$ for $A(L(\Gamma))$.

Equality with $\displaystyle\frac{1}{r}$ requires $s=r$ and $\displaystyle p_\lambda=\frac{1}{r}$ on the
support. The total multiplicity of positive eigenvalues is $r$,
so these are all the positive eigenvalues and each is simple.
Using $C_{qq}=2$ and $\tr C=2m$, the first moment gives
\[
2=C_{qq}=\frac1r\sum_{\lambda>0}\lambda
=\frac{\tr C}{r}=\frac{2m}{r}.
\]
Hence $m=r$. In particular, $C$ is positive definite.
By Lemma~\ref{lem:rank}, the bipartite case has $m=n-1$ and hence
is a tree. The nonbipartite case has $m=n$, so it is unicyclic,
and its unique cycle must be odd.
\end{proof}

The next steps use the second and third local moments at $q$.
In any equality case from Lemma~\ref{lem:bridge-support}, the uniform
spectral weights imply
\begin{align}\label{eq:flat-moments}
(f(C))_{qq}=\frac1m\tr f(C)
\end{align}
for every real function $f$ on the spectrum of $C$. We use
$f(x)=x^2,x^3,x^{-1}$. The inverse is defined because $C$ is positive definite.

\begin{lemma}[A necessary local condition for equality]\label{lem:pendant}
Under the hypotheses of Theorem~\ref{thm:bridge}, equality in
\eqref{eq:bridge-bound} forces $q$ to be a pendant edge.
\end{lemma}

\begin{proof}
Assume equality, and write $q=uv$, $a=d(u)$, and $b=d(v)$.
Let $\tau$ be the number of triangles in $\Gamma$ and set
\[
S_j=\sum_{x\in V(\Gamma)}d(x)^j.
\]
The vertex $q$ of $L(\Gamma)$ has degree $a+b-2$. Since $uv$ is a
bridge, its endpoints have no common neighbor. Thus the triangles
of $L(\Gamma)$ containing $q$ are exactly those formed by three
edges incident with $u$ or by three edges incident with $v$.
Using $C=2I+A(L(\Gamma))$, we obtain
\begin{align}
(C^2)_{qq}&=4+(a+b-2)=a+b+2,\label{eq:local-second}\\
(C^3)_{qq}
&=8+6(a+b-2)+(a-1)(a-2)+(b-1)(b-2)\notag\\
&=a^2+b^2+3(a+b).\notag
\end{align}
The matrices $C$ and $Q$ have the same nonzero eigenvalues with
multiplicities. Expanding $Q=D+A(\Gamma)$ gives
\begin{align}\label{eq:global-moments}
\tr C^2=S_2+2m,\qquad
\tr C^3=S_3+3S_2+6\tau.
\end{align}
For example, $\tr(DA(\Gamma)^2)=S_2$ and
$\tr A(\Gamma)^3=6\tau$.

Combining \eqref{eq:flat-moments} with
\eqref{eq:local-second} to \eqref{eq:global-moments} yields
\begin{align}\label{eq:degree-identities}
S_2=m(a+b),\qquad S_3+6\tau=m(a^2+b^2).
\end{align}
Since $\Gamma$ is connected with at least one edge, all degrees are positive.
Cauchy--Schwarz with weights $d(x)-1\ge0$ gives
\[
\left(\sum_x d(x)(d(x)-1)\right)^2
\le
\left(\sum_x(d(x)-1)\right)
\left(\sum_x d(x)^2(d(x)-1)\right).
\]
Substituting \eqref{eq:degree-identities} and $\sum_xd(x)=2m$
therefore gives
\begin{align}\label{eq:degree-cs}
m^2(a+b-2)^2
\le(2m-n)\left[m(a^2+b^2-a-b)-6\tau\right].
\end{align}

Suppose, for a contradiction, that $a,b\ge2$, and put
\[
X=a^2+b^2-a-b>0,\qquad
Y=(a+b-2)^2.
\]
Then
\begin{align*}
Y-X=2(a-2)(b-2)+(a-2)+(b-2)\ge0,
\end{align*}
with equality if and only if $a=b=2$.

If $\Gamma$ is a tree, then $m=n-1$ and $\tau=0$.
Equation~\eqref{eq:degree-cs} implies
\[
m^2Y\le m(m-1)X<m^2X,
\]
contradicting $Y\ge X$.

If $\Gamma$ is unicyclic, then $m=n$, and
\eqref{eq:degree-cs} implies
\[
Y\le X-\frac{6\tau}{m}\le X.
\]
Thus $a=b=2$, $\tau=0$, and equality must hold in the weighted
Cauchy--Schwarz inequality. Equality there forces all degrees
strictly larger than one to have the same value. That value is
two, since $d(u)=d(v)=2$. Hence every degree is at most two.
A connected unicyclic graph with maximum degree two is a cycle,
which has no bridge. This is another contradiction.

Therefore $\min\{a,b\}=1$, and $q$ is pendant.
\end{proof}

The pendant condition leaves trees and odd unicyclic graphs with a
pendant bridge. The inverse Gram matrix reduces them to the stars
$K_{1,m}$ and the paw, a triangle with one pendant edge. The average
mixing matrix of $K_m$ and the degree identities
\eqref{eq:degree-identities} then exclude all of these except
$P_2$ and $P_3$. For trees and odd unicyclic graphs $B$ has full column rank by
Lemma~\ref{lem:rank}, so $C$ is invertible.

\begin{lemma}[Inverse Gram diagonals]\label{lem:inverse}
Suppose $B$ has full column rank and $C=B^{\mathsf T}B$.
\begin{enumerate}[label=\textup{(\roman*)}]
\item For each edge $f$,
$(C^{-1})_{ff}=\min\{\|z\|^2:B^{\mathsf T}z=\ee_f\}$.
\item If $\Gamma$ is a tree on $n$ vertices and deleting an edge $f$
gives components of orders $s_f$ and $n-s_f$, then
$\displaystyle(C^{-1})_{ff}=\frac{s_f(n-s_f)}{n}$.
\item If $\Gamma$ is unicyclic with an odd cycle, so that $m=n$, then
$(C^{-1})_{ff}=s_f$ for a bridge $f$, where $s_f$ is the order of the
tree component of $\Gamma-f$, and $\displaystyle(C^{-1})_{ff}=\frac{m}{4}$
for an edge $f$ on the cycle.
\end{enumerate}
\end{lemma}

Part (i) is the minimum norm property of the Moore--Penrose inverse \cite[Chapter~3]{BIG2003}. Part (ii) follows from Bapat's tree-incidence formula \cite{Bapat1997}. If $N$ is an oriented incidence matrix of the tree, then $B=D_1ND_2$ for diagonal sign matrices $D_1$ and $D_2$, because a tree is bipartite. Hence
$(B^{\mathsf T}B)^{-1}=D_2(N^{\mathsf T}N)^{-1}D_2$ has the same diagonal as $(N^{\mathsf T}N)^{-1}$.
Part (iii) follows from the inverse-incidence formula of Hessert and Mallik \cite[arXiv version, Theorem~3.1]{HM2021}, also recorded in \cite[arXiv version, Theorem~1.4]{HM2023}. The values in (ii) and (iii) are also the diagonal cases of their formulas for $C^{-1}$ in \cite[arXiv version, Theorems~2.12 and~3.9]{HM2021}. We give direct proofs.

\begin{proof}
(i) The vector $z_f=BC^{-1}\ee_f$ solves $B^{\mathsf T}z_f=\ee_f$ and
lies in $\im B=(\ker B^{\mathsf T})^\perp$. Every solution differs from
$z_f$ by an element of $\ker B^{\mathsf T}$, so $z_f$ is the unique
minimum norm solution, and
$\|z_f\|^2=\ee_f^{\mathsf T}C^{-1}B^{\mathsf T}BC^{-1}\ee_f=(C^{-1})_{ff}$.

(ii) Let $\sigma\in\{1,-1\}^n$ be a bipartition vector of the tree and
let $S$ be a component of $\Gamma-f$, of order $s_f$. The vector $z$
equal to $\sigma$ on $S$ and zero elsewhere satisfies
$B^{\mathsf T}z=\pm\ee_f$, and we fix the sign so that
$B^{\mathsf T}z=\ee_f$. By Lemma~\ref{lem:rank}, $\ker B^{\mathsf T}$
is spanned by $\sigma$, so the orthogonal projection of $z$ onto
$\sigma^\perp$ is the minimum norm solution in (i), and its squared
norm is $\displaystyle s_f-\frac{s_f^2}{n}=\frac{s_f(n-s_f)}{n}$.

(iii) Here $B$ is square and invertible by Lemma~\ref{lem:rank}, so
$B^{\mathsf T}z=\ee_f$ has a unique solution $z$ and (i) gives
$(C^{-1})_{ff}=\|z\|^2$. If $f$ is a bridge, let $z$ vanish on the
component of $\Gamma-f$ containing the cycle and alternate between
$1$ and $-1$ on the tree component, with the sign chosen so that the
entries at the two ends of $f$ sum to one. Then $B^{\mathsf T}z=\ee_f$
and $\|z\|^2=s_f$. If $f$ lies on the cycle of length $\ell$, then
$\Gamma-f$ is a tree in which the two ends of $f$ are joined by a path
of even length $\ell-1$. The equations on $\Gamma-f$ force $z$ to
alternate between $t$ and $-t$ with equal values at the two ends of
$f$, and the equation at $f$ gives $2t=1$. Hence every entry of $z$
has absolute value $\displaystyle\frac{1}{2}$ and
$\displaystyle\|z\|^2=\frac{n}{4}=\frac{m}{4}$.
\end{proof}

\begin{proof}[Proof of Theorem~\ref{thm:bridge}]
The inequality and the displayed values of $\rank B$ follow from
Lemmas~\ref{lem:rank} and \ref{lem:bridge-support}. Suppose equality
holds. By Lemmas~\ref{lem:bridge-support} and \ref{lem:pendant},
the graph is a tree or an odd unicyclic graph, $m=\rank B$, and
$q$ is pendant. Equation~\eqref{eq:flat-moments} gives
\begin{align}\label{eq:inverse-mean}
(C^{-1})_{qq}=\frac1m\tr C^{-1}.
\end{align}

First suppose $\Gamma$ is a tree, so $m=n-1$. Since $q$ is
pendant, Lemma~\ref{lem:inverse}(ii) gives
\[
(C^{-1})_{qq}=\frac{n-1}{n}=\frac mn.
\]
For every edge $f$,
\[
s_f(n-s_f)-(n-1)=(s_f-1)(n-s_f-1)\ge0.
\]
Thus every diagonal entry of $C^{-1}$ is at least $\displaystyle\frac{m}{n}$,
with equality precisely at pendant edges. Equation~\eqref{eq:inverse-mean}
forces every edge to be pendant. Such a tree has diameter at most
two and hence is the star $K_{1,m}$.

Now $L(K_{1,m})=K_m$, and \eqref{eq:complete}, which also holds for
$m=1$ because $\Mhat_A(K_1)=[1]$, gives
\[
\Mhat_{qq}-\frac1m
=1-\frac3m+\frac{2}{m^2}
=\frac{(m-1)(m-2)}{m^2}.
\]
This vanishes exactly for $m=1,2$, that is, for $\Gamma\cong P_2$ and $\Gamma\cong P_3$. Conversely, every edge of $P_2$ and of $P_3$ is a pendant bridge with a bipartite side, $\rank B=m$, and the display gives $\displaystyle\Mhat_{qq}=\frac1m=\frac{1}{\rank B}$.

Next suppose $\Gamma$ is odd unicyclic, with cycle length $\ell$.
Here $n=m$. Since $q$ is pendant,
Lemma~\ref{lem:inverse}(iii) gives $(C^{-1})_{qq}=1$.
Equation~\eqref{eq:inverse-mean} would therefore require
$\tr C^{-1}=m$. On the other hand,
\begin{align*}
\tr C^{-1}
&=\frac{\ell m}{4}+\sum_{\substack{f\in E(\Gamma)\\f\text{ a bridge}}}s_f
\\
&\ge\frac{\ell m}{4}+(m-\ell)
=m+\frac{\ell(m-4)}4.
\end{align*}
The odd cycle has at least three edges and $q$ is not on it, so
$m\ge4$. If $m>4$, the last
quantity is strictly larger than $m$, a contradiction.

If $m=4$, the graph is a triangle with a single pendant vertex, the
paw.
At its unique bridge the endpoint degrees are $a=3$ and $b=1$,
whereas the full degree sequence is $3,2,2,1$. Thus
\[
S_2=3^2+2^2+2^2+1^2=18
\ne 4(3+1)=m(a+b),
\]
contradicting \eqref{eq:degree-identities}. No odd unicyclic graph
can attain equality.
\end{proof}

\begin{remark}[Scope of the two hypotheses]\label{rem:scope}
The inequality in \eqref{eq:bridge-bound} holds for every edge $q$
with $\ee_q\in\im B^{\mathsf T}$, whether or not $q$ is a bridge,
because the proof of Lemma~\ref{lem:bridge-support} only uses this
property, and this is Proposition~\ref{prop:hypergraph} below with
$k=2$. The property cannot be dropped. For the cycle $C_4$ with one
pendant vertex attached at a vertex $w$, the alternating vector of the
cycle lies in $\ker B$, $\rank B=4$, and an exact computation gives
$\displaystyle\Mhat_{qq}=\frac{71}{296}<\frac{1}{4}$ at either cycle edge
$q$ at $w$.

For a bridge $q$, the condition $\ee_q\in\im B^{\mathsf T}$ is equivalent to the existence of a bipartite component of $\Gamma-q$. Lemma~\ref{lem:bridge-support} proves one direction. Conversely, suppose $\ee_q=B^{\mathsf T}z$ and both components are nonbipartite. On every edge $xy\ne q$ we have $z_x+z_y=0$. The odd-cycle argument of Lemma~\ref{lem:rank} forces $z$ to vanish on both components, contradicting $(B^{\mathsf T}z)_q=1$.

The rank bound is the support bound
$\displaystyle\Mhat_{qq}\ge\frac{1}{s}$ of Lemma~\ref{lem:bridge-support}
made explicit through $\rank B$, and the gap can be large. For the
star $K_{1,m}$ with $m\ge2$ the support bound gives $\displaystyle\frac{1}{2}$, the
rank bound gives $\displaystyle\frac{1}{m}$, and the actual value is
$\displaystyle1-\frac{2}{m}+\frac{2}{m^2}$ by \eqref{eq:complete}.
\end{remark}

The rank bound has an incidence formulation for the uniform linear
hypergraphs of Section~\ref{sec:cesaro}. Let $\mathcal H$ be a $k$-uniform linear
hypergraph with $m\ge1$ hyperedges and incidence matrix $B_{\mathcal H}$, and put
\[
C_{\mathcal H}=B_{\mathcal H}^{\mathsf T}B_{\mathcal H}=kI_m+A(L(\mathcal H)),
\qquad r=\rank B_{\mathcal H},
\]
let $P_0$ be the orthogonal projection onto $\ker B_{\mathcal H}$ and set
$\kappa_q=(P_0)_{qq}$ for a hyperedge $q$.

\begin{proposition}[Hypergraph rank bound]\label{prop:hypergraph}
With $\Mhat=\Mhat_A(L(\mathcal H))$, one has
\begin{align}\label{eq:hypergraph}
\Mhat_{qq}\ge
\kappa_q^2+\frac{(1-\kappa_q)^2}{r}.
\end{align}
In particular, if
\begin{align}\label{eq:hypergraph-condition}
\ee_q\in\im B_{\mathcal H}^{\mathsf T},
\quad\text{equivalently}\quad
x_q=0\ \text{for every }x\in\ker B_{\mathcal H},
\end{align}
then $\displaystyle\Mhat_{qq}\ge\frac{1}{r}$.
Under \eqref{eq:hypergraph-condition}, equality with $\displaystyle\frac{1}{r}$
holds if and only if $m=r$, the eigenvalues of $C_{\mathcal H}$
are simple, and all their local spectral weights at $q$ equal $\displaystyle\frac{1}{m}$.
\end{proposition}

\begin{proof}
The scalar shift by $kI_m$ leaves average mixing unchanged.
The weight at the zero eigenvalue of $C_{\mathcal H}$ is $\kappa_q$,
and the weights at its positive eigenvalues sum to $1-\kappa_q$.
There are at most $r$ distinct positive eigenvalues. Cauchy--Schwarz
on their weights proves \eqref{eq:hypergraph}.

Condition~\eqref{eq:hypergraph-condition} is equivalent to
$P_0\ee_q=0$, hence to $\kappa_q=0$. If equality with $\displaystyle\frac{1}{r}$ holds,
the argument in \eqref{eq:rank-cs} shows that there are $r$ simple
positive eigenvalues and each has weight $\displaystyle\frac{1}{r}$. The first moment
then gives
\[
k=(C_{\mathcal H})_{qq}
=\frac1r\tr C_{\mathcal H}
=\frac{km}{r},
\]
so $m=r$. Conversely, the stated spectral conditions give
$\displaystyle \Mhat_{qq}=m\left(\frac{1}{m}\right)^2=\frac{1}{r}$.
\end{proof}

The proof uses only $C_{\mathcal H}=B_{\mathcal H}^{\mathsf T}B_{\mathcal H}$, so \eqref{eq:hypergraph}
holds for the average mixing matrix of the Gram matrix $Z^{\mathsf T}Z$ of any
nonzero real matrix $Z$, with $r=\rank Z$ and $\kappa_q=(P_0)_{qq}$ for the orthogonal projection $P_0$ onto $\ker Z$.

\begin{remark}[Equality in the hypergraph bound]\label{rem:hyperequality}
One $k$-element hyperedge and two $k$-element hyperedges meeting in one vertex have line graphs $K_1$ and $K_2$ and attain equality in \eqref{eq:hypergraph} for every $k\ge2$. The classification $\Gamma\in\{P_2,P_3\}$ of Theorem~\ref{thm:bridge} therefore concerns ordinary graphs.

By Proposition~\ref{prop:hypergraph}, equality under \eqref{eq:hypergraph-condition} with $m\ge3$ requires $\displaystyle\Mhat_{qq}=\frac1m$ in the line graph on $m$ vertices. Conversely, every graph $G$ on $m$ vertices is $L(\mathcal H)$ for such a hypergraph with full column rank. Set $k=\Delta(G)+2$, and let the hyperedge indexed by $x$ consist of the edges incident with $x$ and $k-d(x)$ private vertices. Two distinct hyperedges share only the edge joining their indices, if it exists, so $\mathcal H$ is linear and $L(\mathcal H)=G$. A private vertex in each hyperedge supplies a row $\ee_x^{\mathsf T}$ of $B_{\mathcal H}$, so $r=m$ and \eqref{eq:hypergraph-condition} holds at every hyperedge.

Thus the question is whether a graph on $m\ge3$ vertices can have $\displaystyle\Mhat_{vv}=\frac1m$ at one vertex. An exhaustive exact computation finds no such vertex in a graph with $3\le m\le7$ vertices. This computation only delimits the scope of the results, and no proof depends on it. The question remains open for $m\ge8$.
\end{remark}

\section{An entry criterion and a strongly regular graph corollary}\label{sec:sc}
Throughout this section $G$ is a graph with vertex set $V$, $A=A(G)$, and $\Mhat=\Mhat_A(G)$. Strong cospectrality is defined in \eqref{eq:sc}. Fan and Godsil \cite[arXiv version, Lemma~3.1]{FanGodsil2013} characterize it by
$(E_\theta)_{uu}=(E_\theta)_{vv}=\pm(E_\theta)_{uv}$ for every $\theta$.
Coutinho, Juliano and Spier \cite[Section~2]{CJS} use the condition $(E_\theta)_{uu}=\pm(E_\theta)_{uv}$ for every $\theta$ together with cospectrality.
Godsil's working notes \cite{GodsilNotes} state the following criterion without a separate cospectrality condition and without a proof. We include a proof that uses only the projections and their sum.

\begin{theorem}[Godsil's entry criterion]\label{thm:E}
Vertices $u,v$ are strongly cospectral if and only if
\begin{equation}\label{eq:entry}
\left|(E_\theta)_{uv}\right|=(E_\theta)_{uu}
\qquad\text{for every eigenvalue }\theta.
\end{equation}
\end{theorem}
\begin{proof}
Necessity follows by taking inner products in \eqref{eq:sc}.
Conversely, write $(E_\theta)_{uv}=\varepsilon_\theta(E_\theta)_{uu}$
with $\varepsilon_\theta\in\{-1,1\}$. Then
\begin{align*}
\sum_\theta\|E_\theta\ee_v-\varepsilon_\theta E_\theta\ee_u\|^2
&=\sum_\theta\left((E_\theta)_{vv}-2\varepsilon_\theta(E_\theta)_{uv}+(E_\theta)_{uu}\right)\\
&=\sum_\theta\left((E_\theta)_{vv}-(E_\theta)_{uu}\right)=1-1=0,
\end{align*}
because $\varepsilon_\theta^2=1$ and $\sum_\theta E_\theta=I$. Every summand on the left is nonnegative, so each summand vanishes. This proves \eqref{eq:sc}.
\end{proof}

\begin{remark}[Three entries of the average mixing matrix]\label{cor:three}
Strong cospectrality is also equivalent to
$\Mhat_{uu}=\Mhat_{uv}=\Mhat_{vv}$.
Necessity follows by squaring \eqref{eq:sc} entrywise and summing.
Conversely,
\[
(E_\theta)_{uv}^2\le(E_\theta)_{uu}(E_\theta)_{vv}
\le\frac{(E_\theta)_{uu}^2+(E_\theta)_{vv}^2}{2}.
\]
Equality of the three mixing entries forces equality in each term after summation, giving \eqref{eq:entry}. This criterion is implicit in \cite[arXiv version, proof of Theorem~9.3]{Godsil2013}. It also follows from the average-state Gram identity \cite[Theorem~5]{CGGZ} and \cite[Lemma~5.4]{GS}.
\end{remark}

An exact test over $\Q$ avoids the eigenvalues. Let $\psi$ be the minimal polynomial of $A$ and $\psi=\prod_g g$ its factorization into distinct monic irreducible polynomials $g$ over $\Q$. For each factor $g$ put $\displaystyle\psi_g=\frac{\psi}{g}$. Since $\psi$ has distinct roots, $g$ and $\psi_g$ are coprime, so by the Chinese remainder theorem there is a polynomial $h_g$ with $h_g\equiv1\pmod g$ and $h_g\equiv0\pmod{\psi_g}$, unique modulo $\psi$. Put
\[
 E_g=h_g(A)=\sum_{g(\theta)=0}E_\theta ,
\]
a rational symmetric idempotent that does not depend on the choice of $h_g$, because $\psi(A)=0$.

\begin{corollary}[Rational form]\label{cor:rational}
Vertices $u$ and $v$ are strongly cospectral if and only if $(E_g)_{uv}=\pm(E_g)_{uu}$ for every irreducible factor $g$ of the minimal polynomial of $A$ over $\Q$.
\end{corollary}
This restates, through the rational idempotents, the rational orthogonal involution of Godsil and Smith \cite[Theorem~10.2]{GS}, see also \cite[Lemma~6.8.1]{CGbook}, and the proof uses the same Galois argument.
\begin{proof}
If $(E_g)_{uv}=\varepsilon_g(E_g)_{uu}$ for every $g$, the computation in the proof of Theorem~\ref{thm:E} with $E_g$ in place of $E_\theta$ gives $\sum_g\|E_g\ee_v-\varepsilon_gE_g\ee_u\|^2=0$, so $E_g\ee_v=\varepsilon_gE_g\ee_u$ for every $g$, and applying $E_\theta$ for a root $\theta$ of $g$ gives \eqref{eq:sc}. Conversely let $u,v$ be strongly cospectral. The entries of $E_\theta$ lie in $\Q(\theta)$, and an automorphism $\sigma$ of the splitting field of $\psi$ maps $E_\theta$ to $E_{\sigma(\theta)}$, because $E_\theta$ is the Lagrange polynomial in $A$ at $\theta$. Hence $\sigma$ maps $(E_\theta)_{uu}$ to $(E_{\sigma(\theta)})_{uu}$, so it maps the eigenvalue support $\supp\ee_u$, the set of $\theta$ with $(E_\theta)_{uu}=\|E_\theta\ee_u\|^2\ne0$, onto itself. For $\theta\in\supp\ee_u$, Theorem~\ref{thm:E} gives the sign $\displaystyle\varepsilon_\theta=\frac{(E_\theta)_{uv}}{(E_\theta)_{uu}}\in\{1,-1\}$, and $\sigma$ maps it to $\varepsilon_{\sigma(\theta)}$. A rational number is fixed by $\sigma$, so $\varepsilon_{\sigma(\theta)}=\varepsilon_\theta$. The Galois group acts transitively on the roots of each irreducible factor $g$. Hence either all roots of $g$ lie in $\supp\ee_u$ and $\varepsilon_\theta$ has a common value $\varepsilon_g$ on them, or no root of $g$ lies in $\supp\ee_u$. In the first case, summing $(E_\theta)_{uv}=\varepsilon_g(E_\theta)_{uu}$ over the roots of $g$ gives $(E_g)_{uv}=\varepsilon_g(E_g)_{uu}$. In the second case \eqref{eq:entry} gives $(E_\theta)_{uv}=(E_\theta)_{uu}=0$ for every root $\theta$ of $g$, so $(E_g)_{uv}=(E_g)_{uu}=0$.
\end{proof}

We close with strongly regular graphs. Let $G$ be a connected strongly regular graph with parameters $(n,k,\lambda,\mu)$ that is not complete, so that $\mu\ge1$. The adjacency matrix has the eigenvalue $k$ with idempotent $\displaystyle E_k=\frac{1}{n}J$ and two restricted eigenvalues $r>s$ with
\[
 r+s=\lambda-\mu,\qquad rs=\mu-k,
\]
see \cite[Chapter~10]{GR2001}. Godsil and Smith \cite[Section~6]{GS} prove that a connected strongly regular graph that is not complete multipartite has no strongly cospectral pair. Since $G$ is distance regular of diameter two, the full classification below is a consequence of \cite[Theorem~10.3]{GS} and the involution observation in \cite[arXiv version, Section~4, before Theorem~4.3]{CGGV}. If distinct vertices $u$ and $v$ are strongly cospectral, the former makes the distance-two matrix $T=J-I-A$ a permutation matrix of order two with $T\ee_u=\ee_v$. The latter gives $TE_\theta=\pm E_\theta$ for such a permutation relation $T$. We give a direct proof from the entries.

\begin{corollary}[Strongly regular graphs]\label{thm:D}
Let $G$ be a connected strongly regular graph that is not complete. Distinct vertices $u$ and $v$ are strongly cospectral if and only if $G$ is a cocktail party graph $K_{2,\ldots,2}$ and $\{u,v\}$ is one of its parts, that is, $u$ and $v$ are nonadjacent with $N(u)=N(v)=V\setminus\{u,v\}$. In particular a primitive strongly regular graph has no strongly cospectral pair.
\end{corollary}
\begin{proof}
Since $G$ is regular and not complete, every vertex has $n-k-1\ge1$ non-neighbors. If $u$ and $v$ are strongly cospectral, we show that they are nonadjacent and that $n-k-1=1$. Let $\theta\in\{r,s\}$ have multiplicity $m_\theta$.
The idempotent $E_\theta$ is a polynomial in $A$, so it lies in the span of $I$, $A$ and $J-I-A$, and its diagonal is constant, equal to $\displaystyle\frac{m_\theta}{n}$. The relations $AE_\theta=\theta E_\theta$ and $E_\theta\mathbf 1=0$ give
\[
 (E_\theta)_{uv}=\frac{m_\theta}{n}\cdot
 \begin{cases}
  \dfrac{\theta}{k}, & u\sim v,\\[8pt]
  -\dfrac{1+\theta}{n-k-1}, & u\ne v,\ u\not\sim v.
 \end{cases}
\]
For $\theta=k$ both entries equal $\displaystyle\frac1n$. By Theorem~\ref{thm:E}, $u$ and $v$ are strongly cospectral exactly when the second factor has absolute value one for $\theta=r$ and for $\theta=s$.

If $u\sim v$, then $|r|=|s|=k$. Since $r\ne s$, one of them equals $k$, which is impossible because $k$ is a simple eigenvalue. The adjacent-vertex argument is also in \cite{GodsilNotes}.

If $u\not\sim v$, then $|1+r|=|1+s|=n-k-1$, and $r>s$ forces $1+r=n-k-1$ and $1+s=-(n-k-1)$. Hence $\lambda-\mu=r+s=-2$ and $\mu-k=rs=1-(n-k-1)^2$, so $k-\lambda-1=(n-k-1)^2$. The identity $k(k-\lambda-1)=\mu(n-k-1)$ \cite[Chapter~10]{GR2001} becomes $k(n-k-1)^2=\left(k+1-(n-k-1)^2\right)(n-k-1)$, that is $\left((n-k-1)-1\right)\left(k+(n-k-1)+1\right)=0$. Thus $n-k-1=1$. Every vertex has exactly one non-neighbor, so $G$ is a cocktail party graph and $\{u,v\}$ is one of its parts. Conversely, a cocktail party graph has $n-k-1=1$, $r=0$ and $s=-2$, and then $|1+r|=|1+s|=1$.
\end{proof}

\bigskip
{\bf Acknowledgments.} Theorem~\ref{thm:bridge} grew out of the author's poster \emph{On a new type of quantum random weighted graphs}, presented at the poster session of the 80th Anniversary Meeting of the Korean Mathematical Society in Seoul in June 2026. The poster stated the inequality \eqref{eq:bridge-bound} and asked whether equality holds only for $P_3$ once the trivial case $P_2$ is excluded. Theorem~\ref{thm:bridge} answers that question.

During the preparation of this paper the author became aware of the machine checked proof \cite{PR11860}, posted on 1 October 2026. The results of this paper were obtained independently. The overlap is described after Lemma~\ref{lem:projector}.

\bigskip
{\bf Disclosure statement.} The author does not have any conflicts of interest.

\bigskip
{\bf Data availability.} The exhaustive computation reported in Remark~\ref{rem:hyperequality} examined all graphs with three to seven vertices in exact rational arithmetic. The value in Remark~\ref{rem:scope} was also computed exactly. These computations describe the scope of the results, and no proof depends on them.

\end{document}